\documentclass[12pt]{amsart}
\usepackage{amsmath,amscd, nicefrac, xcolor, setspace, amsfonts, amssymb, bbm, amsthm, tikz, tikz-cd, enumerate, graphicx,verbatim,todonotes,mathtools,url}
\usetikzlibrary{calc,decorations.markings}
\usepackage[margin=1in]{geometry}
\usepackage[shortalphabetic]{amsrefs}
\usepackage{subcaption}
\graphicspath{ {/images/} }
\usepackage{subcaption} 
\usepackage{cite}
\usepackage{graphicx}
\usepackage{textcomp}
\usepackage{comment}
\usepackage{algorithm,algpseudocode}

\usepackage{bbm}

\newtheorem{theorem}{Theorem}[section]

\newtheorem{proposition}{Proposition}[section]
\newtheorem{problem}{Problem}[section]

\newtheorem{remark}{Remark}[section]

\newcommand{\R}{\mathbb{R}}

\newcommand{\map}[3]{#1:#2 \rightarrow #3}

\DeclarePairedDelimiterX{\infdivx}[2]{(}{)}{%
  #1\;\delimsize\|\;#2%
}

\numberwithin{equation}{section}

\title{
Flow Matching for Averaged Systems
}

\author{Daniel Owusu Adu and Yongxin Chen% <-this % stops a space
}

\begin{document}

\maketitle
\thispagestyle{empty}
\pagestyle{empty}

%%%%%%%%%%%%%%%%%%%%%%%%%%%%%%%%%%%%%%%%%%%%%%%%%%%%%%%%%%%%%%%%%%%%%%%%%%%%%%%%
\begin{abstract}
We extend flow matching to ensembles of linear systems in both deterministic and stochastic settings. Averaging over system parameters induces memory leading to a non-Markovian interpolation problem for the stochastic case. In this setting, a control law that achieves the distributional controllability is characterized as the conditional expectation of a Volterra-type control. This conditional expectation in the stochastic settings motivates an open-loop characterization in the noiseless-deterministic setting. Explicit forms of the conditional expectations are derived for special cases of the given distributions and a practical numerical procedure is presented to approximate the control inputs. A by-product of our analysis is a numerical split between the two regimes. For the stochastic case, history dependence is essential and we implement the conditional expectation with a recurrent network trained using independent sampling. For the deterministic case, the flow is memoryless and a feedforward network learns a time-varying gain that transports the ensemble. We show that to realize the full target distribution in this deterministic setting, one must first establish a deterministic sample pairing (e.g., optimal-transport or Schrodinger-bridge coupling), after which learning reduces to a low-dimensional regression in time.
\end{abstract}

%%%%%%%%%%%%%%%%%%%%%%%%%%%%%%%%%%%%%%%%%%%%%%%%%%%%%%%%%%%%%%%%%%%%%%%%%%%%%%%%

\section{Introduction}\label{sec: Introduction}
Flow Matching (FM)~\cite{YL-RTC-HBH-MN-ML:22,MSA-EVE:22,XL-CG-QL:22}, a generative modeling framework also known as continuous normalizing flows~\cite{RTC-YR-JB-DKD:18,WG-RTC-JB-IS-DD:18}, smoothly interpolates between a source distribution $\mu_0 \in \mathcal{P}(\mathbb{R}^d)$ and a target distribution $\mu_f \in \mathcal{P}(\mathbb{R}^d)$ via a continuous-time flow defined by a velocity field. This framework comprises two steps \cite{YL-RTC-HBH-MN-ML:22}: First, choose a probability path $t\mapsto\mu_t$ interpolating between the source $\mu_0$ and target $\mu_f$ distributions and obtain a corresponding vector field $u$ for this interpolation. Secondly, train a vector field using a neural network $v_{\alpha}$, where $\alpha$ represents the learnable parameters, by solving a least-squares regression problem: %by minimizing a discrepancy measure between the velocity $v$ and a displacement vector field $u$: % (such as one arising from optimal transport or Schrodinger bridge problem)
\begin{equation*}%\label{eq: optimization of flow matching}
\min_{v_{\alpha}\in \mathcal{F}}\int_0^{t_f}\mathbb{E}_{x\sim\mu_t}(\|v_{\alpha}(x,t)-u(x,t)\|^2)dt.
\end{equation*}
Here $\mathcal{F}$ is a function class, typically a neural network, that parametrizes the velocity field $v_{\alpha}(x,t)$. Some approaches in selecting the pair $(\mu,u)$ is either through %the continuity equation
%\begin{equation}\label{eq: continuity equation}
%\partial_t\mu_t+\nabla\cdot(v(t,\cdot)\mu_t)=0,    
%\end{equation}
optimal transport (OT)\ Schrodinger bridge (SB) problem~\cite{RWB:07,FAB-EZ:12,YC-TTG-MP:16,YC-TTG-MP:16b,YC:23,PDP:91}. Direct implementation of this optimization is often computationally infeasible.  However, conditioning the loss significantly simplifies the computation~\cite{YL-RTC-HBH-MN-ML:22}. %This approach turns the problem~\eqref{eq: continuity equation} with the boundary conditions into a deep learning-based generative model similar to normalizing flows~\cite{}. Typically, the problem is to find a velocity field $\map{v}{\R^d\times[0,t_f]}{\R^d}$ such that the distribution flow $\mu_t$ evolves according to 
%with boundary conditions $\mu_0$ at initial time and $\mu_{t_f}=\mu_f$ at final time. 

This paper aims to extend the flow matching framework to settings where the interpolation is constrained by a class of ensemble of control systems: %the stochastic differential equation:
\begin{equation}\label{eq:stochastic ensemble of systems}
d X_{\epsilon}(t,\theta)=(A(\theta)X_{\epsilon}(t,\theta)+B(\theta)u_{\epsilon}(t))d t+\sqrt{\epsilon}B(\theta)d W(t).
\end{equation}
Here $\epsilon\geq 0$ is the noise intensity, $(X_{\epsilon}(t,\theta))_{t\in[0,t_f]}$ is an $\R^d$-valued state process of an individual system indexed by  $\theta\in[0,1]$, the quantities $A\in C([0,1];\mathbb{R}^{d\times d})$ and $B\in C([0,1];\mathbb{R}^{d\times m})$ are state and control/noise channels, and $(u_{\epsilon}(t))_{t\in[0,t_f]}$ is an $\mathbb{R}^m$-valued \emph{parameter-independent control process} that is adapted to the filtration generated by the Brownian motion $(W(t))_{t\in[0,t_f]}$. Since the distributions $\mu_0$ and $\mu_f$ are independent of the parameter $\theta$ the natural candidate that alignes with the FM framework is the averaged of~\eqref{eq:stochastic ensemble of systems} over the parameter $\theta\in[0,1]$. However, an important consequence of this averaging is that the resulting averaged process becomes non-Markovian~\cite{DOA-YC:23}. 
%Since the infinitesimal generator of the averaged process is not well-defined due to memory effects in the drift, one cannot directly associate a continuity equation~\eqref{eq: continuity equation} with the averaged process in the usual sense. 
Thus, by interpolating $\mu_0$ to $\mu_f$ using the averaged process, we introduce non-Markovian structure and considerations into the FM framework. This fundamentally extends traditional FM framework and beyond its Markovian control setting recently studied in~\cite{YM-MA-AT-YC:24}, making it suitable for applications involving uncertainty, memory effects, and large-scale generative modeling.

The control of large ensembles, prevalent in various applications including quantum systems, often relies on applying a single control input to all members~\cite{JSL-NK:06,JSL:06,JSL-NK:05,MHL:86}. However, one often relies on optimal control to design a control that interpolates between $\mu_0$ and $\mu_f$~\cite{ADO:22,DOA-YC:24}. In particular, our work in~\cite{ADO:22} studies OT between $\mu_0$ and $\mu_f$ using the averaged system~\eqref{eq:stochastic ensemble of systems}, where $\epsilon=0$. Our work in~\cite{DOA-YC:24} studies this interpolation, where $\epsilon>0$, through the SB framework.  % From an FM point of view, another motivation stem from the fact that the FM is formulated as a local training task 
%making it a computationally efficient and scalable 
%alternative to OT and SB for interpolating between given probability distributions. 
To emphasize on the motivation of this present work, we state here that, even though our work in~\cite{ADO:22} provides an optimal transport map for interpolating between $\mu_0$ and $\mu_f$, this transport map may be discontinuous and sometimes difficult if not impossible to be used in generative modeling. Although our recent work in SB in~\cite{DOA-YC:24} aims to mitigate this issue by providing a stochastic interpolation, it requires a lot of iteration solvers in the context of generative modeling. Furthermore, both frameworks in~\cite{ADO:22,DOA-YC:24} require solving a \emph{global} optimization problem to achieve interpolation. The motivation of the FM framework is to offers a directly \emph{local} trainable smooth flow model, making it highly suitable for generative modeling and high-dimensional applications. % (see Table~\ref{tab:comparison}).  %We state here that along the lines of~\cite{}, the SB method we have developed in~\cite{} can be explored for generative modeling (e.g., diffusion models), but their training can be more complex than flow-based approaches. In contrast to this work, 

Aside the fact that our work also generalizes the Markovian setting in FM framework in~\cite{YM-MA-AT-YC:24} to non-Markovian settings, clearly the tools in~\cite{YM-MA-AT-YC:24} is not readily applicable in our case. For example, the authors in~\cite{YM-MA-AT-YC:24} utilizes the Fokker-Planck equation associated to the Markov process to show equivalence in distribution of a local process and a global process. The memory characteristic of the non-Markov setting makes such quantity ill-suited. 
%Our work addresses the scenario of non-Markovian systems. The non-Markovian nature arises from the ensemble averaging and the structure of the resulting stochastic integrals.  As one will envision, the techniques and approaches in~\cite{YM-MA-AT-YC:24} would not readily extend to such systems. For example, associating the Fokker-Planck equation in our case will be ill-suited. 
We overcome this challenge by directly computing the distributions and showing equivalency through their transition distribution. Also, while in the Markov case the feedback control law that achieves the deterministic and stochastic interpolation is exactly the same form, we will see that this is not true in the non-Markov settings. Computationally, we show that, there is a numerical split between the deterministic and stochastic settings. More precisely, in the stochastic setting, the resulting non-Markovian dynamics can be steered from a given initial distribution to a prescribed target using a recurrent neural network (RNN) (e. g., Long Short-Term Memory (LSTM) network)  or a Transformer trained with importance sampling or independent sampling. In the deterministic setting, the distributional transformation can be achieved with a feedforward network (FNN) by training on pairs produced by a deterministic coupling (e.g., an OT permutation plan), which ensures the learned open-loop control hits the desired terminal distribution. The deterministic coupling help reduce the learning to a one dimensional regression in time. 

%Numerically, while~\cite{YM-MA-AT-YC:24} uses feedforward neural networks (FNNs) that process the state as independent snapshots, FNNs are ill-suited for capturing the history dependence of the control law steering the averaged process. In our approach, due to the nature of the local control, we employ LSTM networks (LSTMs), a type of recurrent neural network (RNN) with memory dependence. By inputting the history of the Brownian motion up to time $t$, the network can sees the full input sequence during training, which allows it to exploit correlations along the sequence to better approximate the conditional expectation at each time step.  The main drawback is that LSTMs increase computational complexity in this non-Markovian setting. In the case where $\epsilon=0$ in~\eqref{eq:stochastic ensemble of systems}, the lack of memory inspires the use of the standard feedforward neural network similar to~\cite{YM-MA-AT-YC:24}.

The structure of the paper is as follows: In Section~\ref{sec: Averaged controllability between Diracs}, we focus on interpolating two points using the averaged system of~\eqref{eq:stochastic ensemble of systems}, for $\epsilon=0$ and $\epsilon> 0$, separately. In Section~\ref{sec: Averaged Controllability between General Distributions}, we discuss the generalization to interpolating between distributions. In Section~\ref{sec: Gaussian and mixture Gaussian initial and target distribution}, we provide an explicit formula for the control law for the special case where the source and the target distributions are Gaussian and mixed Gaussian distributions, respectively. In Section~\ref{sec: Flow Matching Algorithm and Numerical Results}, we discuss the general case for $\epsilon=0$ and $\epsilon> 0$. We employ FM methodology to compute the control that generates their respective flow process that interpolates between any given initial and target distributions. We conclude with a numerical analysis for our flow matching approach.

\section{Averaged controllability between Diracs}\label{sec: Averaged controllability between Diracs}
This section presents a preliminary result concerning interpolating two points using the averaged of an ensemble of systems. We present both deterministic and stochastic cases.
\subsection{Deterministic Case}
We consider the ensemble of linear systems given by~\eqref{eq:stochastic ensemble of systems}, where $\epsilon=0$, initialized at $X(0,\theta)=x_0$.
%\begin{align}\label{eq:ensemble of system}
%\dot{X}(t,\theta)=&A(\theta)X(t,\theta)+B(\theta)u(t),\cr
%,%\quad\text{where $\theta\in [0,1]$},
%\end{align}
Note that the state of an individual system indexed by $\theta$ at time $t$ is characterized by 
\begin{align}\label{eq:ensemble of state}
X(t,\theta)=e^{A(\theta)t} x_0+\int_{0}^{t}e^{A(\theta)(t-\tau)}B(\theta)u(\tau)d\tau.
\end{align}

\begin{problem}\label{prob: 1}
Given any pair $(x_0,x_f)\in\R^d\times\R^d$, find a parameter-independent control input $u\in L^1([0,t_f];\mathbb{R}^m)$ such that the ensemble of states in~\eqref{eq:ensemble of state} satisfies $\int_{0}^{1} x(t_f,\theta)d\theta=x_f$.  
\end{problem}

Given any pair $(x_0,x_f)\in\R^d\times\R^d$, if such a control input exists, we say that the ensemble of system~\eqref{eq:stochastic ensemble of systems}, where $\epsilon=0$, is averaged controllable. The following result guarantees the existence of such controls. 
\begin{theorem}\cite[Theorem~1]{ZE:14}\label{thm:rank condition}
The ensemble of systems~\eqref{eq:stochastic ensemble of systems}, where $\epsilon=0$, is said to be averaged controllable  if and only if the vector space spanned by the columns of $\left\{\int_{0}^{1}A(\theta)^kB(\theta)d\theta \right\}_{k= 0}^{\infty}$ is of rank $d$.
\end{theorem}
Following from~\cite{ADO:22}, this is simplified to the invertibility of the following matrix:
\begin{align}\label{eq: Gramian}
G_{t_f,0}:=&\int_{0}^{t_f}\varphi(t_f,\tau)\varphi(t_f,\tau)^{\mathrm{T}}d\tau,
\end{align}
where
\begin{equation}\label{eq: convolution function}
\Phi(t_f,\tau)=\int_{0}^{1}e^{A(\theta)(t_f-\tau)}B(\theta)d\theta    
\end{equation}
is the convolution kernel. Since this kernel is not a transition matrix, unlike standard control problem, the above matrix $G_{t_f,0}$ is not the controllability Gramian.

\begin{proposition}\label{prop: Interpolation of determinsstic case}
Given any pair $z=(x_0,x_f)\in\R^d\times\R^d$. Suppose $G_{t_f,0}$ defined in~\eqref{eq: Gramian} is invertible. Then, a control that solves Problem~\ref{prob: 1} is characterized as 
\begin{equation}\label{eq: averaged control}
u^z(t)=\Phi(t_f,t)^{\mathrm{T}} G_{t_f,0}^{-1}\left(x_f-\left(\int_{0}^{1}e^{A(\theta)t_f}d\theta\right) x_0\right).
\end{equation}
The resulting interpolation:
\begin{equation}\label{eq: Ensemble interpolation}
X^z(t,\theta)=e^{A(\theta)t} x_0+\int_{0}^{t}e^{A(\theta)(t-\tau)}B(\theta)\Phi(t_f,\tau)^{\mathrm{T}}d\tau G_{t_f,0}^{-1}\left(x_f-\left(\int_{0}^{1}e^{A(\theta)t_f}d\theta\right) x_0\right).   
\end{equation}
Moreover, the control in~\eqref{eq: averaged control} minimizes the $\mathrm{L}^2$-norm $\int_0^{t_f}\|u(t)\|^2dt$ among all admissible control.
\end{proposition}
\begin{remark}
Before presenting the proof, suppose that $A(\theta)=A$, for all $\theta\in[0,1]$ in~\eqref{eq:stochastic ensemble of systems}, where $\epsilon=0$. Under this assumption, the matrix in~\eqref{eq: Gramian} simplifies to the classical controllability Gramian:
\begin{equation}\label{eq: Actual Control_Gramian}
\tilde{G}_{t_f,0}:=\int_{0}^{t_f}\tilde{\Phi}(t_f,\tau)\tilde{\Phi}(t_f,\tau)^{\mathrm{T}}d\tau,    
\end{equation}
where   
\[
\tilde{\Phi}(t_f,\tau)=e^{A(t_f-\tau)}\tilde{B}
\]
is the simplified convolution kernel in~\eqref{eq: convolution function} and
\[
\tilde{B}:=\int_{0}^{1}B(\theta)d\theta
\]
is the averaged control channel. Consequently, the control law~\eqref{eq: averaged control} reduces to:
\begin{align}\label{eq: classic control}
u^z(t)=&\tilde{\Phi}(t_f,t)^{\mathrm{T}} \tilde{G}_{t_f,0}^{-1}\left(x_f-e^{At_f} x_0\right)
\end{align}
and $X^z(t,\theta)=X^z(t)$, where
\begin{equation*}%\label{eq: Ensemble interpolation}
X^z(t)=e^{At} x_0+\int_{0}^{t}e^{A(t-\tau)}\tilde{B}\tilde{\Phi}(t_f,\tau)^{\mathrm{T}}d\tau \tilde{G}_{t_f,0}^{-1}\left(x_f-e^{At_f} x_0\right).   
\end{equation*}
It is well-known, see for instance~\cite[pp.~137]{RWB:15} that, if the matix~\eqref{eq: Actual Control_Gramian} is invertible then, the control in~\eqref{eq: classic control} minimizes the $\mathrm{L}^2$-norm $\int_0^{t_f}\|u(t)\|^2dt$ among all admissible control for the linear system:
\begin{equation}\label{eq: classical linear system}
\dot{X}(t)=AX(t)+\tilde{B}u(t),\quad X(0)=x_0\text{ and }\quad X(t_f)=x_f. 
\end{equation}
We state here that if the matix~\eqref{eq: Actual Control_Gramian} is invertible then one says that the pair $(A,\tilde{B})$ in the system~\eqref{eq: classical linear system} is controllable. That is the matrix $\begin{bmatrix}
\tilde{B}\hspace{3mm} A\tilde{B}\hspace{3mm} A^2\tilde{B}\hspace{3mm}\dots\hspace{3mm}A^{d-1}\tilde{B}    
\end{bmatrix}$ is of rank $d$, see~\cite[Theorem~5]{EBL-LM:67}. Note that this result is different from Theorem~\ref{thm:rank condition}, in that one requires, in general, checking this for all $k\in\mathbb{Z}_+$ as opposed to $k$ running from $0$ to $d-1$ in the latter case.
    
\end{remark}

\begin{proof}
From~\cite{ADO:22} and~\cite[Theorem~3]{ZE:14}, we have that the control input with  minimum $\mathrm{L}^2$-norm  is 
\begin{equation}\label{eq: optimal control}
u^*(t)=\int_{0}^{1}B^{\mathrm{T}}(\theta)\varphi^*(t,\theta)d\theta,
\end{equation}
where $\varphi^*$ is the solution to the adjoint system
the adjoint system
\begin{align}\label{eq:adjoint system}
\dot{\varphi^*}(t,\theta)=&-A^{\mathrm{T}}(\theta)\varphi^*(t,\theta),\cr%\quad 0<t<t_f,\cr
\varphi^*(t_f,\theta)=&\varphi_f^*%\quad\text{where $\theta\in [0,1]$}
\end{align}
where $\varphi^*_f$ solves 
\begin{align}\label{eq:functional}
\min_{\varphi_f\in\R^d} J(\varphi_f):&=\frac{1}{2}\int_{0}^{t_f}\left |\int_{0}^{1} B^{\mathrm{T}}(\theta)\varphi(t,\theta)d\theta\right |^2dt -\langle x_f,\varphi_f\rangle+\left\langle x_0,\int_{0}^{1}\varphi(0,\theta)d\theta\right\rangle.
\end{align}
By substituting 
\[
\varphi(t,\theta)=e^{A^{\mathrm{T}}(\theta)(t_f-t)}\varphi_f\quad\text{and}\quad \varphi(0,\theta)=e^{A^{\mathrm{T}}(\theta)t_f}\varphi_f
\]
in~\eqref{eq:functional}, we have that %the critical point is characterized as: 
\[
\varphi_f^*=G_{t_f,0}^{-1}\left(x_f-\left(\int_{0}^{1}e^{A(\theta)t_f}d\theta\right) x_0\right).
\]
Hence 
\[
\varphi^*(t,\theta)=e^{A^{\mathrm{T}}(\theta)(t_f-t)}G_{t_f,0}^{-1}\left(x_f-\left(\int_{0}^{1}e^{A(\theta)t_f}d\theta\right) x_0\right).
\]
Therefore, we have that~\eqref{eq: averaged control} hold. By substituting~\eqref{eq: averaged control} in~\eqref{eq:ensemble of state} we get~\eqref{eq: Ensemble interpolation}. This completes the proof.
\end{proof}

Note that the averaged interpolation $x^z(t):=\int_0^1X^z(t,\theta)d\theta$ is characterized as
\begin{equation}\label{eq: averaged interpolation}
x^z(t)=\left(\int_0^1e^{A(\theta)t}d\theta\right) x_0+\left(\int_{0}^{t}\Phi(t,\tau)\Phi(t_f,\tau)^{\mathrm{T}}d\tau\right) G_{t_f,0}^{-1}\left(x_f-\left(\int_{0}^{1}e^{A(\theta)t_f}d\theta\right) x_0\right).    
\end{equation}

\subsection{Stochastic Case}
Consider the ensemble of systems~\eqref{eq:stochastic ensemble of systems} with the presence of noise intensity $\epsilon>0$,
initialized at $X_{\epsilon}(0,\theta)=x_0$, for all $\theta\in[0,1]$. 

The problem of interest is similarly stated as follows:
\begin{problem}\label{problem: stochastic}
Given any pair $(x_0,x_f)\in\R^d\times\R^d$, find a parameter-independent control process $u\in L^2([0,t_f];\mathbb{R}^m)$ such that the ensemble of states in~\eqref{eq:stochastic ensemble of systems} initialized at $X(0,\theta)=x_0$, for all $\theta\in[0,1]$, satisfy $\int_{0}^{1} X(t_f,\theta)d\theta=x_f$ almost surely.    
\end{problem}

Problem~\ref{problem: stochastic} aim to construct a stochastic bridge for the uncontrolled process characterized as
\begin{equation}\label{eq: pass_ave_at_t}
y_{\epsilon}(t)=\left(\int_0^1e^{A(\theta)t}d\theta\right) x_0+ \sqrt{\epsilon}\int_{0}^{t}\Phi(t,\tau) d W(\tau), 
\end{equation}
conditioned that 
\[
y_{\epsilon}(t_f)=x_f \quad\text{almost surely.}
\]
As described in~\cite{DOA-YC:23}, since $\int_0^1e^{A(\theta)t}d\theta$ is not a transition matrix we have that the averaged processes $(y_{\epsilon}(t))_{t\in[0,t_f]}$ in~\eqref{eq: pass_ave_at_t} and $\left(x_{\epsilon}(t)=\int_{0}^{1} X_{\epsilon}(t,\theta)d\theta\right)_{t\in[0,t_f]}$, corresponding to~\eqref{eq:stochastic ensemble of systems}, are both Volterra process with memory and hence non-Markovian process. The stochastic bridge has been developed in~\cite{DOA-YC:23} using a stochastic optimal control formulation. We state the result without proof.

\begin{proposition}
Given any $z=(x_0,x_f)\in\R^d\times\R^d$. Suppose $G_{t_f,t}$, for all $0\leq t<t_f$, in~\eqref{eq: Gramian} is invertible. Then, a conditional control that solves Problem~\ref{problem: stochastic} is characterized as 
\begin{equation}\label{eq: optimal stochastic control}
  u_{\epsilon}^z(t)= -\sqrt{\epsilon}\int_0^t\Phi(t_f,t)^T G_{t_f,\tau}^{-1}\Phi(t_f,\tau) dW(\tau)+\Phi(t_f,t)^{T}G_{t_f,0}^{-1}\left(x_f-\left(\int_{0}^{1}e^{A(\theta)t_f}d\theta\right) x_0\right).    
 \end{equation} 
 where $\Phi(t_f,\tau)$ is defined in~\eqref{eq: convolution function}.
The resulting conditional state process interpolation:
\begin{multline}\label{eq: averaged stochastic interpolation}
 x_{\epsilon}^z(t)=\left(\int_0^1e^{A(\theta)t}d\theta\right) x_0-\sqrt{\epsilon}\int_0^t\int_0^{\tau}\Phi(t,\tau)\Phi(t_f,\tau)^T G_{t_f,s}^{-1}\Phi(t_f,s) dW(s)d\tau\\+\left(\int_0^t\Phi(t,\tau)\Phi(t_f,\tau)^{T}d\tau\right) G_{t_f,0}^{-1}\left(x_f-\left(\int_{0}^{1}e^{A(\theta)t_f}d\theta\right) x_0\right) +\sqrt{\epsilon}\int_0^t \Phi(t,\tau) dW(\tau).
\end{multline}
Moreover, the control in~\eqref{eq: optimal stochastic control} minimizes the expectation of the $\mathrm{L}^2$-norm $\int_0^{t_f}\|u(t)\|^2dt$ among all admissible control.
\end{proposition}
\begin{remark}
    Note that if $\epsilon=0$ then $u_{\epsilon}^z$ and $x_{\epsilon}^z$ in~\eqref{eq: optimal stochastic control} and~\eqref{eq: averaged stochastic interpolation}, respectively, reduce to the deterministic case $u^z$ and $x^z$ in~\eqref{eq: averaged control} and~\eqref{eq: averaged interpolation}, respectively (see Figure~\ref{fig:ou_interpolation} and Figure~\ref{fig:anti_damped_interpolation}). Also, we have shown in~\cite{DOA-YC:23,DOA-YC:24} that if $A(\theta)=A$ then the Volterra process or stochastic feedforward in~\eqref{eq: optimal stochastic control} reduces to the Markov process
\begin{align}\label{eq: opt_pinn_con}
u^*_{\epsilon}(x^z_{\epsilon}(t),t)=-\tilde{B}^{\mathrm{T}}e^{A^{\mathrm{T}}(t_f-t)}G_{t_f,t}^{-1}\left(\left(e^{A(t_f-t)}\right)x^z_{\epsilon}(t)-x_f\right),
\end{align}
where 
\begin{equation}\label{eq: Markov_state_process}
d x_{\epsilon}(t)=(Ax_{\epsilon}(t,\theta)+\tilde{B}u^*_{\epsilon}(t))d t+\sqrt{\epsilon}\tilde{B}d W(t) 
\end{equation}
is a Markov state process, see~\cite{CY-GT:15}. We state here that under the condition $A(\theta)=A$ and $\epsilon=0$, the control~\eqref{eq: classic control} admits the feedback form~\eqref{eq: opt_pinn_con}. Therefore, the feedback control law that achieves the deterministic and stochastic interpolation is exactly the same, under the condition $A(\theta)=A$, see~\cite{YM-MA-AT-YC:24}. 
\end{remark}

\begin{figure}[!t]
    \centering
    % -------- First row: Ornstein–Uhlenbeck (damped) --------
    \begin{subfigure}[b]{0.45\textwidth}
        \centering
        \includegraphics[width=\textwidth]{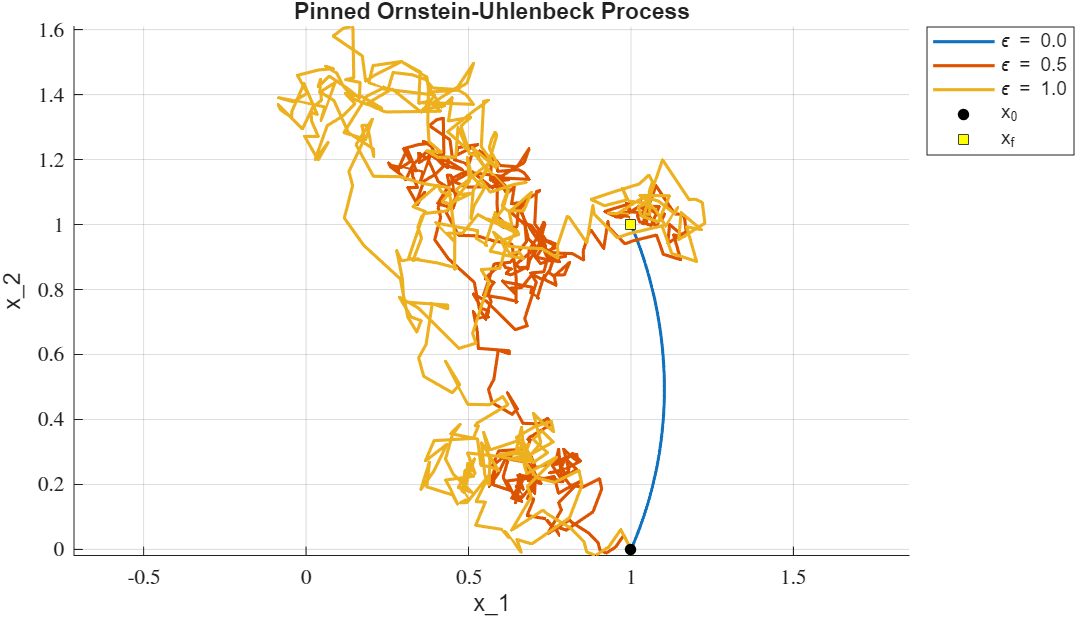}
        \caption{Pinned trajectories of the averaged Ornstein-Uhlenbeck process in the $(x_1,x_2)$-plane.}
        \label{fig:ou_2D}
    \end{subfigure}
    \hfill
    \begin{subfigure}[b]{0.45\textwidth}
        \centering
        \includegraphics[width=\textwidth]{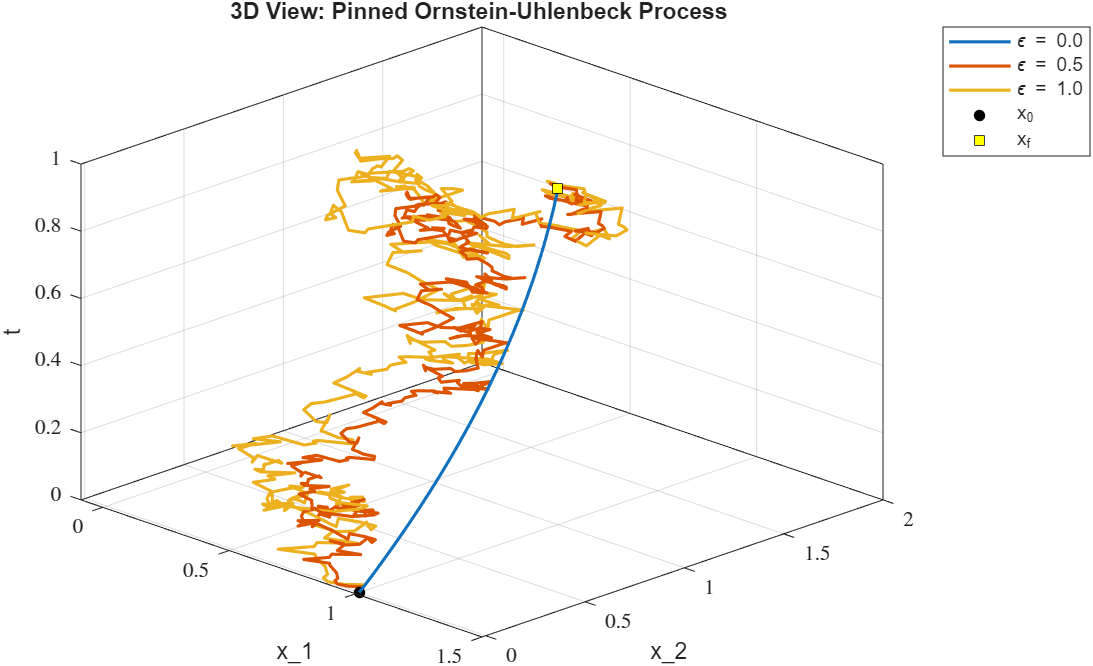}
        \caption{Spatio-temporal evolution of the pinned Ornstein--Uhlenbeck process in $(x_1,x_2,t)$.}
        \label{fig:ou_3D}
    \end{subfigure}
    \caption{Interpolation of fixed endpoints $x_0=[1,0]$ and $x_f=[1,1]$ by the averaged of linear systems in~\eqref{eq:stochastic ensemble of systems} governed by the 
    Ornstein-Uhlenbeck dynamics parameterized as~\eqref{eq: 2D-Ornstein-Uhelbeck_parameters}. 
    Each color corresponds to a different noise intensity $\epsilon \in \{0,0.5,1\}$.}
    \label{fig:ou_interpolation}
\end{figure}

% ===========================================================

\begin{figure}[!t]
    \centering
    % -------- Second row: Anti-damped rotational system --------
    \begin{subfigure}[b]{0.45\textwidth}
        \centering
        \includegraphics[width=\textwidth]{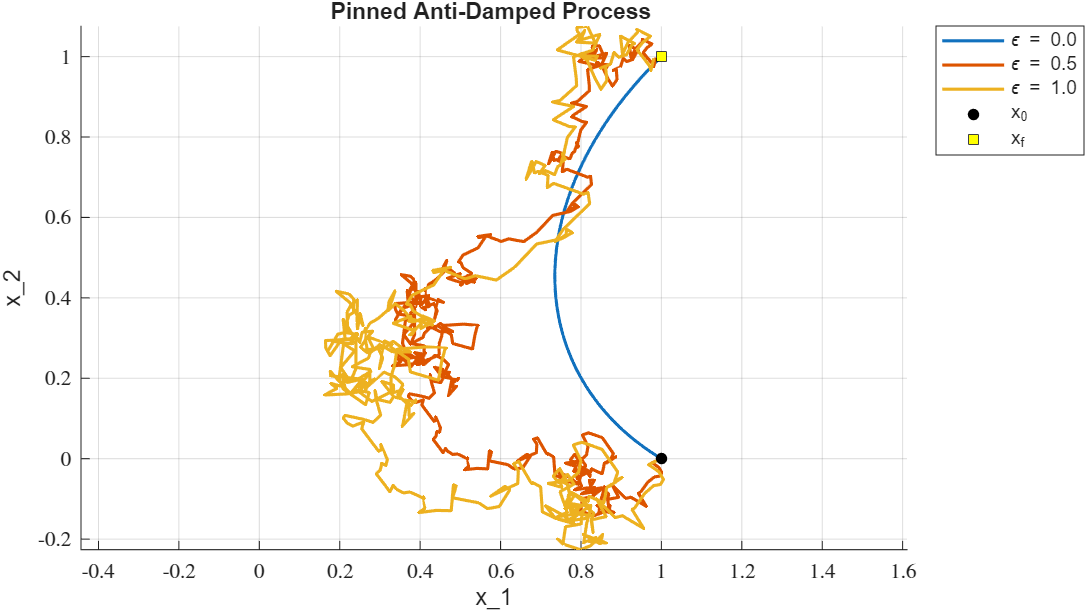}
        \caption{Pinned trajectories of the averaged anti-damped rotational process in $(x_1,x_2)$-plane.}
        \label{fig:ad_2D}
    \end{subfigure}
    \hfill
    \begin{subfigure}[b]{0.45\textwidth}
        \centering
        \includegraphics[width=\textwidth]{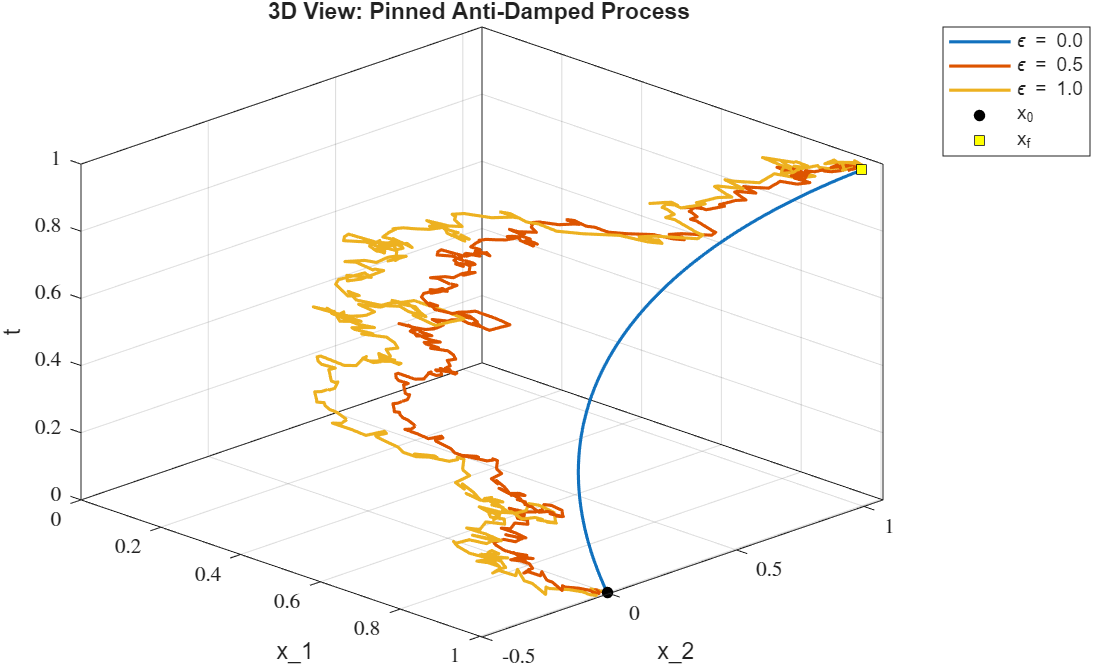}
        \caption{Spatio-temporal evolution of the anti-damped rotational process in $(x_1,x_2,t)$-plane.}
        \label{fig:ad_3D}
    \end{subfigure}
    \caption{Interpolation of fixed endpoints $x_0=[1,0]$ and $x_f=[1,1]$ by averaged of linear systems in~\eqref{eq:stochastic ensemble of systems} characterized by the parameters in~\eqref{eq: 2D_anti-damped_parameters}. 
    Compared with the Ornstein-Uhlenbeck case, trajectories exhibit outward spiralling and non-reverting behavior.}
    \label{fig:anti_damped_interpolation}
\end{figure}
% describing a rotationally expanding (anti-damped) drift field

%We consider  in $\R^2$. That is $t_f=1$ and~\eqref{eq: ens_passive_dyn}, where ~\ref{fig:optimal_control_process_again} ({\color{red} work on the figure labels later})\ref{fig:2D_optimal_averaged_process} ({\color{red} work on the figure labels later})
Figure~\ref{fig:ou_interpolation} and Figure~\ref{fig:anti_damped_interpolation} show the interpolations of the averaged of a controlled Ornstein–Uhlenbeck and an anti-damped processes over the time horizon $[0,1]$ with minimum $L_2$-norm.
%Figure~$1\mathrm{A}$ correspond to the one-dimensional control with minimum $L_2$-norm and the corresponding state process with parameters
%\begin{equation}\label{eq: 1D-parameter}
%A(\theta)=-\theta\quad\text{and}\quad B(\theta)=1.
%\end{equation}
Figure~\ref{fig:ou_interpolation} corresponds to the two-dimensional case, with parameters:
\begin{equation}\label{eq: 2D-Ornstein-Uhelbeck_parameters}
A(\theta)=\begin{bmatrix}
    0&-\theta\\
    \theta&0
\end{bmatrix}\quad\text{and}\quad B(\theta)=\begin{bmatrix}
    1&0\\
    0&1
\end{bmatrix}.    
\end{equation}
Figure~\ref{fig:anti_damped_interpolation} correspond to the two-dimensional case, with parameters:
\begin{equation}\label{eq: 2D_anti-damped_parameters}
A(\theta) = 
\begin{bmatrix}
\sin(\theta) & \cos(\theta) \\
-\cos(\theta)& \sin(\theta)
\end{bmatrix}, \quad 
B(\theta) = 
\begin{bmatrix}
    0&-\theta\\
    \theta& 0
\end{bmatrix}    
\end{equation}

\section{Averaged Controllability between General Distributions}\label{sec: Averaged Controllability between General Distributions}
Our aim in this section is to use the deterministic and stochastic interpolations in~\eqref{eq: averaged interpolation} and~\eqref{eq: averaged stochastic interpolation} to design a control that steers a given initial distribution $\mu_0\in \mathcal{P}(\R^d)$ to a target distribution $\mu_f\in\mathcal{P}(\R^d)$. Throughout, we assume that $\mu_0$ and $\mu_f$ are independent distributions. % through~\eqref{eq:ensemble of system} and~\eqref{eq:stochastic ensemble of systems}. 

%{\color{blue} We will do the deterministic case here later.}

\begin{problem}\label{problem: general stochastic}
Let $\epsilon\geq 0$. Given the pair $(\mu_0,\mu_f)\in\mathcal{P}(\R^d)\times\mathcal{P}(\R^d)$, find a parameter-independent control process $u\in L^2([0,t_f];\mathbb{R}^m)$ that steers the ensemble of states in~\eqref{eq:stochastic ensemble of systems} from $X(0,\theta)\sim\mu_0$, for all $\theta\in[0,1]$, to $\int_0^1X(t_f,\theta)d\theta\sim\mu_f$. %satisfies $\int_{0}^{1} X(t_f,\theta)d\theta=x_f$ almost surely.    
\end{problem}

For the noiseless case, we have the following result:
\begin{theorem}\label{thm: deterministic control}
Let $\epsilon=0$.  Suppose $G_{t_f,0}$ in~\eqref{eq: Gramian} is invertible. Then a candidate control that solves Problem~\ref{problem: general stochastic} is characterized as the process
\begin{equation}\label{eq: proposed noiseless control process}
u(t)=u^z(t),    
\end{equation}
where $u^z(t)$ is defined in~\eqref{eq: averaged control} and $z=(x_0,x_f)\sim\mu_0\otimes\mu_f$.
\end{theorem}
\begin{proof}
Let $z=(x_0,x_f)\sim\mu_0\otimes\mu_f$,  then the distribution flow $P_t$ of the noiseless averaged process 
\begin{equation}\label{eq: noiseless averaged process}
x(t)=\left(\int_0^1e^{A(\theta)t}d\theta\right) x_0+\left(\int_{0}^{t}\Phi(t,\tau)u(\tau)d\tau\right) 
\end{equation}
is 
\begin{equation}\label{eq: distribution with dirac}
P_t(A)=\int_{\R^d\times\R^d}\mathbbm{1}_A\left(\left(\int_0^1e^{A(\theta)t}d\theta\right) x_0+\left(\int_{0}^{t}\Phi(t,\tau)u(\tau)d\tau\right)\right)d\mu_0(x_0)d\mu_f(x_f),    
\end{equation}
for all measurable sets $A\subset\R^d\times\R^d$. Let 
\[
T_t(x_0,x_f):=\left(\int_0^1e^{A(\theta)t}d\theta\right) x_0+\left(\int_{0}^{t}\Phi(t,\tau)u(\tau)d\tau\right). 
\]
If~\eqref{eq: proposed noiseless control process} holds, then %$u(\tau)$ is linear in $(x_0,x_f)$ and hence 
\[
T_t(x_0,x_f)=\left(\int_0^1e^{A(\theta)t}d\theta\right) x_0+\left(\int_{0}^{t}\Phi(t,\tau)\Phi(t_f,\tau)^{\mathrm{T}}d\tau\right) G_{t_f,0}^{-1}\left(x_f-\left(\int_{0}^{1}e^{A(\theta)t_f}d\theta\right) x_0\right)
\]
invertible and hence~\eqref{eq: distribution with dirac} is equivalent to 
\begin{equation}\label{eq: distribution flow}
P_t(A)=(\mu_0\otimes\mu_f)(T_{t}^{-1})(A).
\end{equation}
From Proposition~\ref{prop: Interpolation of determinsstic case}, since
\[
T_0(x_0,x_f)=x_0\quad\text{and}\quad T_{t_f}(x_0,x_f)=x_f
\]
hold, we have that  the control process in~\eqref{eq: proposed noiseless control process} \emph{generates} the probability flow~\eqref{eq: distribution with dirac} that interpolates between the given marginals  
%the distribution flow~\eqref{eq: distribution flow} satisfies the marginal constraints 
\[
P_0=\mu_0\quad\text{and}\quad P_{t_f}=\mu_f.
\]
%This proves that the control in~\eqref{eq: proposed noiseless control process} steers the distribution~\eqref{eq: distribution with dirac} of the noiseless averaged process~\eqref{eq: noiseless averaged process} from $x(0)\sim\mu_0$ at time $t=0$ to $x(t_f)\sim\mu_f$.
This completes the proof.
\end{proof}

%We state here that in the case where both $\mu_0,\mu_f$ are fully observed, Problem~\ref{problem: general stochastic} has been solved in~\cite{ADO:22,DOA-YC:24} using optimization technique. For the purpose stated in the introduction, we will use flow matching technique to solve Problem~\ref{problem: general stochastic}. 
For the noisy case, we consider the stochastic bridge $x_{\epsilon}^z$ defined in~\eqref{eq: averaged stochastic interpolation} satisfying the marginal constraints:
\begin{equation}\label{eq: marginal bridge}
x_{\epsilon}^z(0)=x_0\sim\mu_0\quad\text{and}\quad x_{\epsilon}^z(t_f)=x_f\sim\mu_f.  \end{equation}
We solve Problem~\ref{problem: general stochastic} by finding a control in the system~\eqref{eq:stochastic ensemble of systems} so that the probability law of the averaged process $\left(x(t)=\int_0^1X(t,\theta)d\theta\right)_{t\in[0,t_f]}$ equilibrates that of the probability law of stochastic bridge $(x_{\epsilon}^z(t))_{t\in[0,t_f]}$ with marginal constraints~\eqref{eq: marginal bridge}. The following result provides a candidate of the control that solves Problem~\ref{problem: general stochastic}
\begin{theorem}
Let $\epsilon>0$ and $\mathcal{F}_t=\sigma(W(s); 0\leq s\leq t)$ be the filtration generated by the Brownian motion.  Suppose $G_{t_f,t}$, for all $0\leq t<t_f$, in~\eqref{eq: Gramian} is invertible. Then the control that solves Problem~\ref{problem: general stochastic} is characterized as
\begin{equation}\label{eq: proposed control}
u_{\epsilon}(t)=\mathbb{E}(u_{\epsilon}^z(t)|\mathcal{F}_t)    
\end{equation}
where $u_{\epsilon}^z$ where is defined in~\eqref{eq: optimal stochastic control} be the corresponding control and the expectation is taken over the distribution of the random variable $z=(x_0,x_f)\sim\mu_0\otimes\mu_f$ given the filtration $\mathcal{F}_t$.
\end{theorem}
\begin{remark}
We state here that $\mathcal{F}_t=\sigma(W(s); 0\leq s\leq t)=\sigma(x_{\epsilon}^z(s); 0\leq s\leq t)$, where $x_{\epsilon}^z(t)$ is the Volterra process with memory characterized in~\eqref{eq: averaged stochastic interpolation}. Therefore, under the condition $A(\theta)=A$, since the Volterra process in~\eqref{eq: averaged stochastic interpolation} reduces to the memoryless process in~\eqref{eq: Markov_state_process}, we have that the control process in~\eqref{eq: opt_pinn_con} is a Markov process. Thus if $x_{\epsilon}(t)=x$ a.s then, we have that control $u^*_{\epsilon}(t)=u^*_{\epsilon}(t,x)$ where
\begin{align}\label{eq: cond_opt_pinn_con}
u^*_{\epsilon}(t,x)=-\tilde{B}^{\mathrm{T}}e^{A^{\mathrm{T}}(t_f-t)}G_{t_f,t}^{-1}\left(\left(e^{A(t_f-t)}\right)x-x_f\right).
\end{align}
Hence~\eqref{eq: proposed control} reduces to $u_{\epsilon}(t)=u_{\epsilon}(t,x)$
where
\[
u_{\epsilon}(t,x)=\mathbb{E}(u_{\epsilon}^z(t,x)| x_{\epsilon}(t)=x).
\]
which coincides with the result in~\cite{YM-MA-AT-YC:24}. %One should observe that in the case $A(\theta)=A$, the control law has the same form for any $\epsilon\geq 0$. In the general case, this is not true.
\end{remark}
\begin{proof}

 Firstly, for a fixed $z=(x_0,x_f)$, we show that the probability law of the controlled process 
\begin{equation}\label{eq: output_with_IC}
 x_{\epsilon}(t)=\left(\int_0^1e^{A(\theta)t}d\theta\right) x_0+\int_0^t \Phi(t,\tau)(u_{\epsilon}(\tau)d\tau +\sqrt{\epsilon} dW(\tau)),    
\end{equation}
is equal to the probability law of the stochastic bridge $x_{\epsilon}^z(t)$ in~\eqref{eq: averaged stochastic interpolation}. To this end,  using~\eqref{eq: proposed control}, the probability distribution of the process~\eqref{eq: output_with_IC} is a Gaussian distribution with mean
\begin{equation}\label{eq: controlled mean 1}
m_{\epsilon}(t)=\left(\int_0^1e^{A(\theta)t}d\theta\right) x_0+\mathbb{E}\int_0^t \Phi(t,\tau)\mathbb{E}(u_{\epsilon}^z(\tau)|\mathcal{F}_{\tau})d\tau    
\end{equation}
with covariance matrix $\epsilon G_{t,0}$. Also, the probability distribution of the process~\eqref{eq: averaged stochastic interpolation} is a Gaussian distribution with mean
\begin{equation}\label{eq: mean of stochastic bridge}
m_{\epsilon}^z(t)=\left(\int_0^1e^{A(\theta)t}d\theta\right) x_0+\mathbb{E}\int_0^t \Phi(t,\tau)u_{\epsilon}^z(\tau)d\tau    
\end{equation}
with the same covariance matrix $\epsilon G_{t,0}$. However, from~\eqref{eq: optimal stochastic control}, using the tower property of conditional expectation, we have that the mean value at time $t$ in~\eqref{eq: mean of stochastic bridge} is the same as  
\begin{equation}\label{eq: conditional mean of stochastic bridge}
m_{\epsilon}^z(t)=\left(\int_0^1e^{A(\theta)t}d\theta\right) x_0+\mathbb{E}\int_0^t \Phi(t,\tau)\mathbb{E}(u_{\epsilon}^z(\tau)|\mathcal{F}_{\tau})d\tau.    
\end{equation}
This implies that both processes in~\eqref{eq: output_with_IC}, under the control~\eqref{eq: proposed control} and~\eqref{eq: averaged stochastic interpolation} have the same transition probabilities. 

Therefore, following from~\eqref{eq: marginal bridge}, we conclude that the probability law of $ x^z(t)$ is equal to the probability law of $x(t|x_0)$ for all $t \in [0, t_f]$. In particular, $x(t_f|x_0)\sim\mu_f$. This concludes the proof.
\end{proof}

\vspace{-9.0mm}

\section{Gaussian and mixture Gaussian initial and target distribution}\label{sec: Gaussian and mixture Gaussian initial and target distribution}
For the noiseless case, we have provided an explicit formula~\eqref{eq: proposed noiseless control process} for arbitrary initial distribution $\mu_0\in\mathcal{P}(\R^d)$ and target distribution $\mu_f\in\mathcal{P}(\R^d)$. The goal of this section is to derive an analytical formula for the noisy control in~\eqref{eq: proposed control} when the initial and target distributions are either Gaussian or Gaussian mixture distributions. As mentioned in the introduction, the choice of a Gaussian mixture for the target distribution is motivated by its relevance to flow matching.

Consider
\begin{equation}\label{eq: special distributions}
\mu_0=\mathcal{N}(m_0,\Sigma_0)\quad\text{and}\quad \mu_f=\sum_{i=1}^L\omega_i\mathcal{N}(m_i,\Sigma_i)    
\end{equation}
where $\sum_{i=1}^L\omega_i=1$. Given $x_0$ and $W(\tau)$, where $0\leq\tau\leq t$, since 
\begin{multline}\label{eq: noisy control}
\mathbb{E}(u_{\epsilon}^z(t)|\mathcal{F}_t)\\=-\sqrt{\epsilon}\int_0^t\Phi(t_f,t)^T G_{t_f,\tau}^{-1}\Phi(t_f,\tau) dW(\tau)+\Phi(t_f,t)^{T}G_{t_f,0}^{-1}\left(\mathbb{E}(x_f|\mathcal{F}_t)-\left(\int_{0}^{1}e^{A(\theta)t_f}d\theta\right) x_0\right),    
\end{multline}
our goal reduces to finding the formula for $\mathbb{E}(x_f|\mathcal{F}_t)$, where at time $t$ the $\sigma$-algebra $\mathcal{F}_t$ is determined by~\eqref{eq: averaged stochastic interpolation} which we rearrange to the form:
\begin{equation}\label{eq: rearranged averaged process}
 x_{\epsilon}^z(t)=Y(t)x_0+Z(t)x_f+R_{\epsilon}(t),
\end{equation}
where
\[
Y(t)=\left(\int_0^1e^{A(\theta)t}d\theta\right)-\left(\int_0^t\Phi(t,\tau)\Phi(t_f,\tau)^{T}d\tau\right) G_{t_f,0}^{-1}\left(\int_{0}^{1}e^{A(\theta)t_f}d\theta\right)
\]
and
\[
Z(t)=\left(\int_0^t\Phi(t,\tau)\Phi(t_f,\tau)^{T}d\tau\right) G_{t_f,0}^{-1}
\]
are deterministic functions and
\[
R_{\epsilon}(t):=\sqrt{\epsilon}\int_0^t\Phi(t,\tau)\left(dW(\tau)-\int_0^{\tau}\Phi(t_f,\tau)^T G_{t_f,s}^{-1}\Phi(t_f,s) dW(s)d\tau \right).
\]
is a noisy process. Since at time $t$, we have that $W(\tau)$, where $0\leq\tau\leq t$ is given, we have that $R_{\epsilon}(t)$ is a Gaussian memory process with mean 
\[
\mathbb{E}(R_{\epsilon}(t)):=-\sqrt{\epsilon}\int_0^t\int_0^{\tau}\Phi(t,\tau)\Phi(t_f,\tau)^T G_{t_f,s}^{-1}\Phi(t_f,s) dW(s)d\tau
\]
and covariance matrix $\epsilon G_{t,0}$.
We proceed to the following result.
\begin{theorem}
Consider Problem~\ref{problem: general stochastic} where $\mu_0$ and $\mu_f$ are given in~\eqref{eq: special distributions}. Then a candidate that solves Problem~\ref{problem: general stochastic} is~\eqref{eq: noisy control} where
\begin{equation}\label{eq: calculation of the conditional expectation}
\mathbb{E}(x_f|\mathcal{F}_t)=\frac{1}{\sum_{i=1}^L\tilde{\omega}_i(t)}\sum_{i=1}^L\tilde{\omega}_i(t)(m_i+\Gamma_i(t)\left(x^z(t)-\chi_i(t)\right)    
\end{equation}
with 
\begin{align*}
\tilde{\omega}_i(t):=&\omega_i\exp{\left(-\frac{1}{2}\left(x^z(t)-\chi_i(t)\right)^{\mathrm{T}}\left(Y(t)\Sigma_0 Y(t)^{\mathrm{T}}+Z(t)\Sigma_i Z(t)^{\mathrm{T}}+\epsilon G_{t,0}\right)^{-1}\left(x^z(t)-\chi_i(t)\right)\right)},\\
\Gamma_i(t):=&\Sigma_iZ(t)^{\mathrm{T}}\left(Y(t)\Sigma_0 Y(t)^{\mathrm{T}}+Z(t)\Sigma_i Z(t)^{\mathrm{T}}+\epsilon G_{t,0}\right)^{-1},\\
\chi_i(t):=&\left(Y(t)m_0+Z(t)m_i+\mathbb{E}(R_{\epsilon}(t))\right).
\end{align*}
\end{theorem}

\begin{proof}
Consider the special case where $\mu_0=\mathcal{N}(m_0,\Sigma_0)$ and $\mu_f=\mathcal{N}(m_f,\Sigma_f)$, then using the formula
\begin{equation*}
\mathbb{E}(x_f|\mathcal{F}_t)=\mathbb{E}(x_f)+\mathrm{Cov}(x_f,x^z(t))(\mathrm{Cov}(x^z(t),x^z(t)))^{-1}(x^z(t)-\mathbb{E}(x^z(t))),  
\end{equation*}
where $x^z(t)$ is in~\eqref{eq: rearranged averaged process}, we have that
\begin{align*}
\mathbb{E}(x_f|\mathcal{F}_t)= m_f+\Sigma_fZ(t)^{\mathrm{T}}\left(Y(t)\Sigma_0 Y(t)^{\mathrm{T}}+Z(t)\Sigma_f Z(t)^{\mathrm{T}}+\epsilon G_{t,0}\right)^{-1}\\ \left(x^z(t)-\left(Y(t)m_0+Z(t)m_f+\mathbb{E}(R_{\epsilon}(t))\right)\right).
\end{align*}
This concludes the formula in~\eqref{eq: calculation of the conditional expectation} for $L=1$. The generalization $L>1$ follows from~\cite{YM-MA-AT-YC:24}. This finishes the proof.
\end{proof}

\section{Flow Matching Algorithm and Numerical Results}\label{sec: Flow Matching Algorithm and Numerical Results}
The analytical determination of the conditional expectation in~\eqref{eq: proposed control} is generally intractable.  Consequently, we employ a flow matching approach to obtain a numerical approximation of~\eqref{eq: proposed control} by solving: % the following least-squares regression problem:
\begin{multline}\label{eq: loss function of flow matching}
\min_{f_{\epsilon}\in \mathcal{F}}\int_0^{t_f}\mathbb{E}_{z\sim\mu_0\otimes\mu_f}\left(\left\|f_{\epsilon}\left(x_0,t,\sqrt{\epsilon}\int_{0}^{t}\Phi(t,\tau) d W(\tau)\right)-u_{\epsilon}^z(t)\right\|^2\right)dt\\\approx \min_{f_{\epsilon}\in \mathcal{F}}\frac{1}{N}\sum^N_{i=1}\int_0^{t_f}\left(\left\|f_{\epsilon}\left(x^i_0,t,\sqrt{\epsilon}\int_{0}^{t}\Phi(t,\tau) d W(\tau)\right)-u_{\epsilon}^{z^i}(t)\right\|^2\right)dt,
\end{multline}
where $\epsilon> 0$ 
The expectation is approximated using $N$ independent samples $z^i = (x^i_0, x^i_f) \sim \mu_0 \otimes \mu_f$, $i \in {1, \dots, N}$. %In this section, we handle both the deterministic and stochastic case. In both cases, the algorithm has two stages. In the training stage and the prediction stage.

In the case where $\epsilon=0$, by the orthogonal projection property of conditional expectation in $L^2$, the optimizer for minimum square error:
\[
u\left(x^i_0,t\right)=\underset{f_{0}\in \mathcal{F}}{\mathrm{argmin}}\mathbb{E}\left(\left\|f_{0}\left(x^i_0,t\right)-u_{0}^{z^i}(t)\right\|^2\right)dt,
\]
for all $i\in 1,\dots,N$, where $u_{0}^{z}:=u^z$ is defined in~\eqref{eq: averaged control} is characterized as the conditional mean
\[
u\left(x^i_0,t\right)= \mathbb{E}(u_{0}^{z^i}(t)|x_0^i,t)
\]
(see e.g.,\cite[pp~85]{DW:91} or~\cite[pp~475]{PB:95}) and hence simplifies to
\begin{equation*}
u\left(x^i_0,t\right)= \Phi(t_f,t)^{\mathrm{T}} G_{t_f,0}^{-1}\left(\mathbb{E}_{x^i_f\sim\mu_f}(x^i_f)-\left(\int_{0}^{1}e^{A(\theta)t_f}d\theta\right) x_0^i\right),   
\end{equation*}
since $(x^i_0,x^i_f)\sim\mu_0\otimes\mu_f$. Therefore, an open-loop control $u(x_0,\cdot)$ trained against teacher controls $u_{0}^{z^i}(\cdot)$ using independent pairings $(x^i_0,x^i_f)$ cannot reproduce the desired spread and covariance of $\mu_f$. It only steers the state in~\eqref{eq: noiseless averaged process} from $\mu_0$ to its mean $\mathbb{E}_{x_f\sim\mu_f}(x_f)$. Here $x_f$ is a random variable taking variables in $\{x^i_f\}$ with probability $\frac{1}{N}$. %Therefore, given the initial and desired distributions $\mu_0,\mu_f$, respectively, the trained open-loop controls obtained from product couplings $\mu_0 \otimes \mu_f$ steers from $\mu_0$ to the mean of $\mathbb{E}_{x^i_f\sim\mu_f}(x^i_f)$ and not the full $\mu_f$. 
To steer to the full $\mu_f$, for the case where $\epsilon=0$, we rather approximate via optimal transport coupling:
\begin{equation}\label{eq: Regression with OT_coupling}
\min_{f_{0}\in \mathcal{F}}\int_0^{t_f}\mathbb{E}_{z\sim\pi^{\star}}\left(\left\|f_{0}\left(x_0,t\right)-u_{0}^{z}(t)\right\|^2\right)dt.    
\end{equation}
Here $\pi^{\star}$ has marginal distribution $\mu_0$ and $\mu_f$ and is obtained by solving an optimal transport problem with quadratic cost (see Algorithm~\ref{alg:flow_matching_deterministic_ot_coupling}). In this case the optimizer in~\eqref{eq: Regression with OT_coupling} is characterized as:
\begin{equation}\label{eq: control_via_OT_coupling}
u(x_0,t)= \Phi(t_f,t)^{\mathrm{T}} G_{t_f,0}^{-1}\left(T(x_0)-\left(\int_{0}^{1}e^{A(\theta)t_f}d\theta\right) x_0\right).    
\end{equation}
Here $T(x_0):=\mathbb{E}_{\pi^{\star}}(x_f|x_0)=x_f$ is the transport map. In this case given $(x_0,T(x_0)=x_f)$ the control in~\eqref{eq: control_via_OT_coupling} steers the state in~\eqref{eq: noiseless averaged process} from $x_0\sim\mu_0$ to $x(t_f)=T(x_0)\sim\mu_f$. This gives us two ways to train the open-loop control. The first is to directly train against the teacher control $\Phi(t_f,t)^{\mathrm{T}} G_{t_f,0}^{-1}\left(T(x_0)-\left(\int_{0}^{1}e^{A(\theta)t_f}d\theta\right) x_0\right)$ after one obtains an optimal transport permutation. The second is to only train $K(t)$ against the gain matrix $\Phi(t_f,t)^{\mathrm{T}} G_{t_f,0}^{-1}$ and after multiply $u(x_0,t)\approx K(t)\left(T(x_0)-\left(\int_{0}^{1}e^{A(\theta)t_f}d\theta\right)x_0\right)$. While the first trains against a $d+1$-dimensional regression in space-time, the latter reduces the training to a one-dimensional regression in time.

In connection to Theorem~\ref{thm: deterministic control}, since $u=u^{z}$ yields equal distribution in their respective flows, this implies using an independent coupling, one must rather train $u(t,\Delta)$, where $\Delta := x_f-\left(\int_{0}^{1}e^{A(\theta)t_f}d\theta\right) x_0$, (even though not an open-loop control) to successfully drive the given initial distribution to the desired distribution. 

For both $\epsilon>0$ and $\epsilon=0$, the numerical computation is performed in two stages, the training stage and the prediction stage. For the \textbf{training stage}, in the case where $\epsilon>0$, the memory is very important. This motivates the use of an RNN or Transformer to take care of the memory.  In particular, for $N=1000$, we use an LSTM network with Adam optimizer and a piecewise learning rate schedule was trained over $100$ epochs ($1.262\times 10^5$ iterations). In the case where $\epsilon=0$, there is no reason to use such sophisticated network architecture.  In particular, for $N=1000$, we use an FNN network with two hidden layers (64 units each, trained with the scaled conjugate gradient algorithm) on deterministic optimal transport coupling between $\{x_0^i\}_{i=1}^{1000}$ to $\{x_f^i\}_{i=1}^{1000}$ sampled independently from $\mu_0$ and $\mu_f$ respectively. For the \textbf{prediction stage}, for the case where $\epsilon>0$, we first sample $\{x_0^i\}_{i=1}^{1000}$ and  use the Euler-Maruyama method, with $\triangle t=0.001$, to simulate $1000$ independent realization of the averaged process~\eqref{eq: output_with_IC} using the trained non-anticipating control law learned in the training stage. For the purpose of visualization we only show $500$ sample paths. For the case where $\epsilon=0$, we repeat the same process but use the trapezoid method for numerical integration to simulate the deterministic flow.

%$\{x_0^i\}_{i=1}^{1000}$ to $\{x_f^{\pi^*(i)}i\}_{i=1}^{1000}$ 

%, we train an LSTM network to approximate the local control. The network has an LSTM layer with 128 hidden units, followed by fully connected layers with $64, 32$, and $2$ neurons. ReLU activation is used after the first two fully connected layers. The Adam optimizer is employed with an initial learning rate of $e^{-3}$, a piecewise decay factor of $0.5$ every $50$ epochs, and a gradient threshold of $1$. The network is trained for $100$ epochs on a sample size of $N=1000$ with a mini-batch size of $64$.  

%In the \textbf{prediction stage}, using the Euler-Maruyama method, the learned control obtained in the training stage is used to simulate $N = 500$ independent realizations of sample trajectories $\{x_\epsilon^i(t)\}_{i=1}^{N}$. However, $500$ sample paths are plotted for the purpose of visualization.

\begin{algorithm}[H]
\caption{Flow Matching for Stochastic Averaged Systems (Product Coupling)}
\label{alg:flow_matching}
\begin{algorithmic}[1]
 \Statex \textbullet~\textbf{Initialize given parameters:} 
\State Initial distribution $\mu_0$, target distribution $\mu_f$ (possibly a mixture of Gaussians)
\State System parameters $A(\theta)$, $B(\theta)$, $\epsilon> 0$ and $(W(t))_{0\leq t\leq t_f}$ in $$d X_{\epsilon}(t,\theta)=(A(\theta)X_{\epsilon}(t,\theta)+B(\theta)u_{\epsilon}(t))d t+\sqrt{\epsilon}B(\theta)d W(t)$$.
\Statex \textbullet~\textbf{Use $A(\theta)$ and $B(\theta)$ to compute relevant functions:}
\State Compute the following:
 %   \begin{enumerate}
         \begin{equation*}
        \Phi(t,\tau) =\int_0^1e^{A(\theta)(t-\tau)}B(\theta)d\theta\quad\text{and}\quad G_{t_f,t}=\int_{t}^{t_f}\Phi(t_f,\tau)\Phi(t_f,\tau)^{\mathrm{T}}d\tau
    \end{equation*}
%    \end{enumerate}
 \Statex \textbullet~\textbf{Generate Interpolating Distribution:}
%\State \textbf{}
\State \hspace{\algorithmicindent} Obtain sample pairs $z^i = (x^i_0, x^i_f) \sim \mu_0 \otimes \mu_f$, $i \in {1, \dots, N}$. For each $z^i = (x^i_0, x^i_f)$ pair, compute
the conditional stochastic feedforward control process $(u_{\epsilon}^{z^i}(t))_{0\leq t\leq t_f}$ using pre-computed functions in Step~3, $A(\theta)$ and $\epsilon> 0$ in step~1:% and relevant functions in step~2:
\begin{equation*}
  u_{\epsilon}^{z^i}(t)= -\sqrt{\epsilon}\int_0^t\Phi(t_f,t)^T G_{t_f,\tau}^{-1}\Phi(t_f,\tau) dW(\tau)+\Phi(t_f,t)^{T}G_{t_f,0}^{-1}\left(x^i_f-\left(\int_{0}^{1}e^{A(\theta)t_f}d\theta\right) x^i_0\right),   
 \end{equation*}
 where  $i \in {1, \dots, N}$.
%\State \hspace{\algorithmicindent} \textbf{end For}
\Statex \textbullet~\textbf{Use a Neural Network (e.g., an LSTM) to Learn the Control Law $u_{\epsilon}(t)$:}
\State \hspace{\algorithmicindent} %Choose the appropriate neutral network 
Consider an appropriate neural network $\mathcal{F}$ and define a function class $f_{\epsilon}\in \mathcal{F}$ to approximate $u_{\epsilon}(t)=\mathbb{E}(u_{\epsilon}^z(t)|\mathcal{F}_t)$. Use the input samples $\left(x^i_0,t,\sqrt{\epsilon}\int_{0}^{t}\Phi(t,\tau) d W(\tau)\right)$, as the training data to train a function $f_{\epsilon} \in \mathcal{F}$ using the regression problem:
\begin{equation*}%\label{eq: optimization of flow matching}
u_{\epsilon}(t)\approx\underset{f_{\epsilon}\in \mathcal{F}}{\mathrm{argmin}}\frac{1}{N}\sum^N_{i=1}\int_0^{t_f}\left(\left\|f_{\epsilon}\left(x^i_0,t,\sqrt{\epsilon}\int_{0}^{t}\Phi(t,\tau) d W(\tau)\right)-u_{\epsilon}^{z^i}(t)\right\|^2\right)dt.
\end{equation*}
\Statex \textbullet~\textbf{Simulate the process:}
\State \hspace{\algorithmicindent} The trained control process $u_{\epsilon}(t)$ is used to steer the process:
\[
x_{\epsilon}(t)=\left(\int_0^1e^{A(\theta)t}d\theta\right) x_0+\int_0^t \Phi(t,\tau)(u_{\epsilon}(\tau)d\tau +\sqrt{\epsilon} dW(\tau)).
\]
\end{algorithmic}
\end{algorithm}
%{Trained_final.eps}

\begin{algorithm}[H]
\caption{Flow Matching for Deterministic Averaged Systems (OT Coupling)}
\label{alg:flow_matching_deterministic_ot_coupling}
\begin{algorithmic}[1]
\Statex \textbf{Inputs:} Initial law $\mu_0$, target law $\mu_f$; system families $A(\theta),B(\theta)$, $\theta\in[0,1]$; final time $t_f=1$.
%\Statex \textbf{Output:} Deterministic open-loop field $u(x_0,t)$ and terminal trajectories $\{x(t)\}_{0\le t\le 1}$.

\Statex\textbf{(I) OT coupling with quadratic cost}
\State Draw i.i.d.\ samples $\{x_0^i\}_{i=1}^N\sim\mu_0$ and $\{\tilde x_f^j\}_{j=1}^N\sim\mu_f$.
\State Form the cost matrix $C_{ij}=\|x_0^i-\tilde x_f^j\|^2$ and solve the assignment problem:
\[
\pi^\star \;=\;\arg\min_{\pi\in S_N}\;\sum_{i=1}^N C_{i,\pi(i)}.
\]
\State Set $x_f^i:=\tilde x_f^{\,\pi^\star(i)}$ for $i=1,\dots,N$ (this is the discrete OT pairing).

\Statex\textbf{(II) Teacher control}
\State For each $t_j\in[0,t_f]$ and $z^i=(x_0^i,x_f^i)$, define the teacher control:
\[
u^{z^i}(t_j)\;=\;\Phi(t_f,t_j)^{\!\top}\,G_{0,t_f}^{-1}\,(x_f^i - \left(\int_0^1e^{A(\theta)t_f}d\theta\right)\,x_0^i).
\]
\Statex\textbf{(III) Learning of the open-loop field}
\State Build the dataset on the grid $\{t_j\}$:
\[
\mathcal{D}
=\big\{\big([x_0^i,t_j],\;u^{z^i}(t_j)\big)\;\big|\; i=1,\dots,N,\;j=0,\dots,M\big\}.
\]
\State Fit $u_{\mathrm{train}}\in\mathcal{F}$ (e.g., a feedforward network) by least squares:
\[
u_{\mathrm{train}}(x_0^i,t_j)=\arg\min_{f\in\mathcal{F}} \frac{1}{N(M{+}1)}
\sum_{i=1}^N\sum_{j=0}^{M}\big\|\,f(x_0^i,t_j)-u^{z^i}(t_j)\,\big\|^2.
\]

\Statex\textbf{(IV) Deterministic rollout with the learned control}
\State For each $i= 1,\dots, N$, propagate:
\[
x^i(t_j)\;=\;\left(\int_0^1e^{A(\theta)t_j}d\theta\right)\,x_0^i
\;+\;\sum_{k=0}^{j}\Phi(t_j,t_k)\,u_{\mathrm{train}}(x_0^i,t_k)\Delta t.
\]
\end{algorithmic}
\end{algorithm}

%%%%%%%%%%%%%%%%%%%%%%%%%%%%%%%%%%%%%
\begin{figure}[h!]
    \centering
    % --- Left: Baseline (Ornstein-Uhlenbeck) ---
    \begin{subfigure}[t]{0.48\textwidth}
        \centering
        \includegraphics[width=\linewidth,height=5.5cm,keepaspectratio]{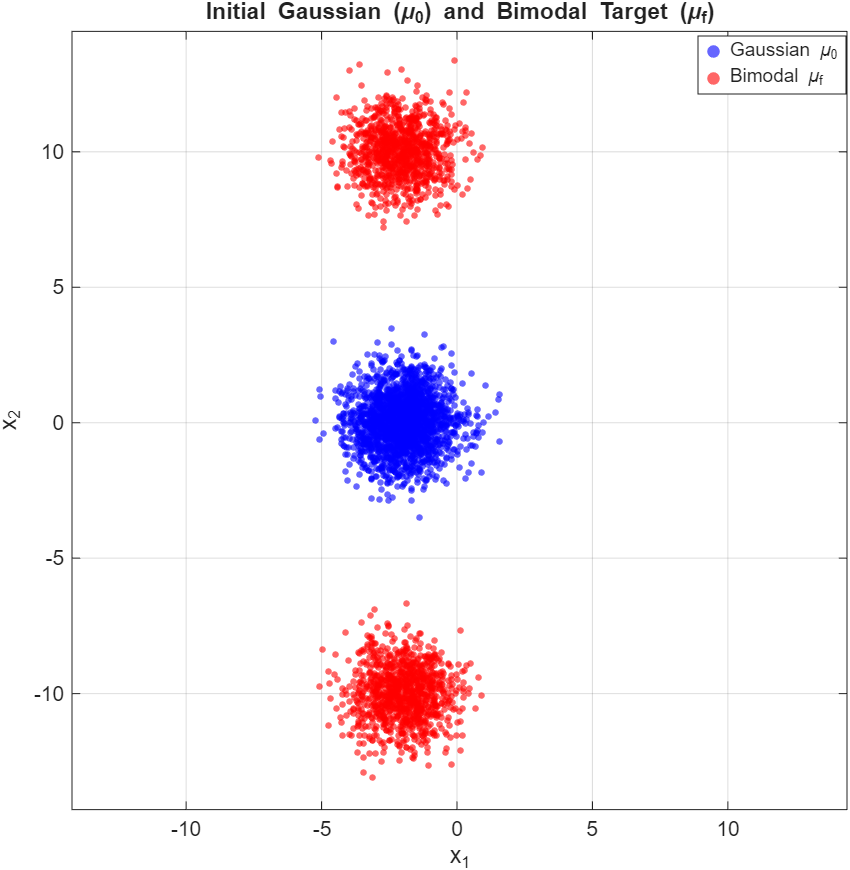}
        \caption{Samples from the Ornstein-Uhlenbeck process. The blue cloud represents samples from the starting Gaussian distribution $\mu_0$, while the red cloud represents samples drawn from the target mixture of Gaussians $\mu_f$.}
        \label{fig:initial_final_ornstein-uhlenbeck}
    \end{subfigure}
    \hfill
    % --- Right: Anti-Damped Variant ---
    \begin{subfigure}[t]{0.48\textwidth}
        \centering
        \includegraphics[width=\linewidth,height=5.5cm,keepaspectratio]{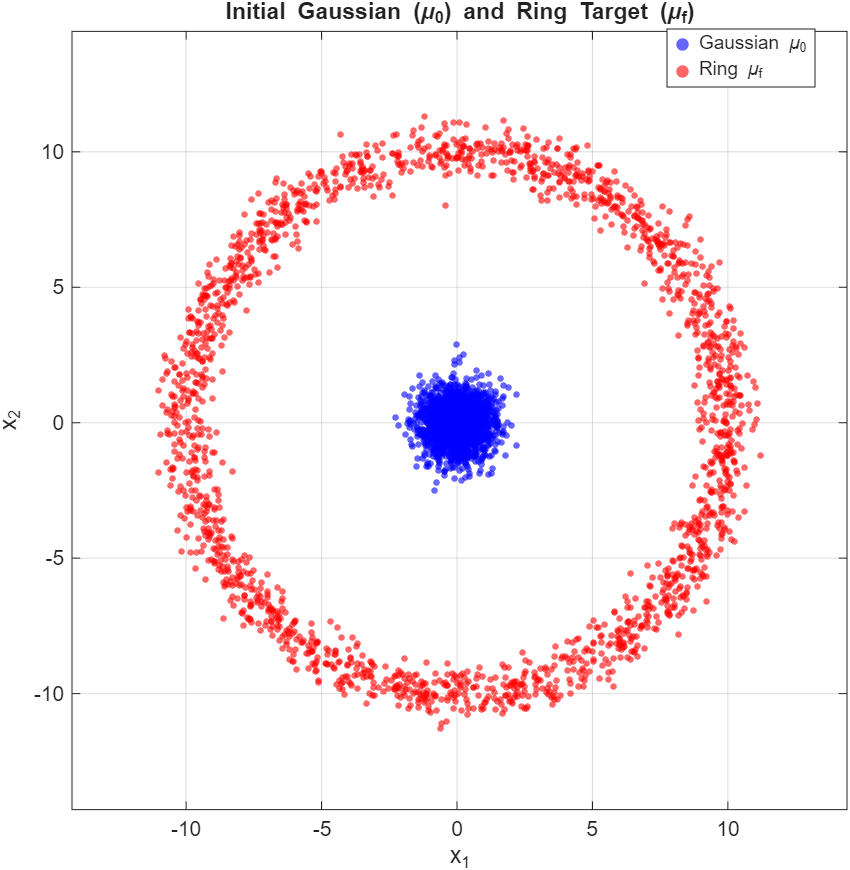}
        \caption{Samples from the anti-damped ensemble process. The blue cloud represents $\mu_0$ (initial Gaussian), while the red cloud represents $\mu_f$ (ring-like target).}
        \label{fig:initial_final_anti-damped}
    \end{subfigure}

    \caption{Comparison of initial and final distributions for the ensemble of Ornstein-Uhlenbeck and anti-damped processes. Both subplots are scaled to equal size for visual comparison.}
    \label{fig:distribution_comparison}
\end{figure}

%%%%%%%%%%%%%%%%%%%%%%%%%%%%%%%%%%%%%%%%
%~\eqref{eq:stochastic ensemble of systems} governed by the Ornstein-Uhlenbeck dynamics parameterized as~\eqref{eq: 2D-Ornstein-Uhelbeck_parameters}

\begin{figure}[h!]
    \centering
    % --- Left: baseline (Ornstein--Uhlenbeck) ---
    \begin{subfigure}[t]{0.46\textwidth}
        \centering
        \includegraphics[width=\textwidth]{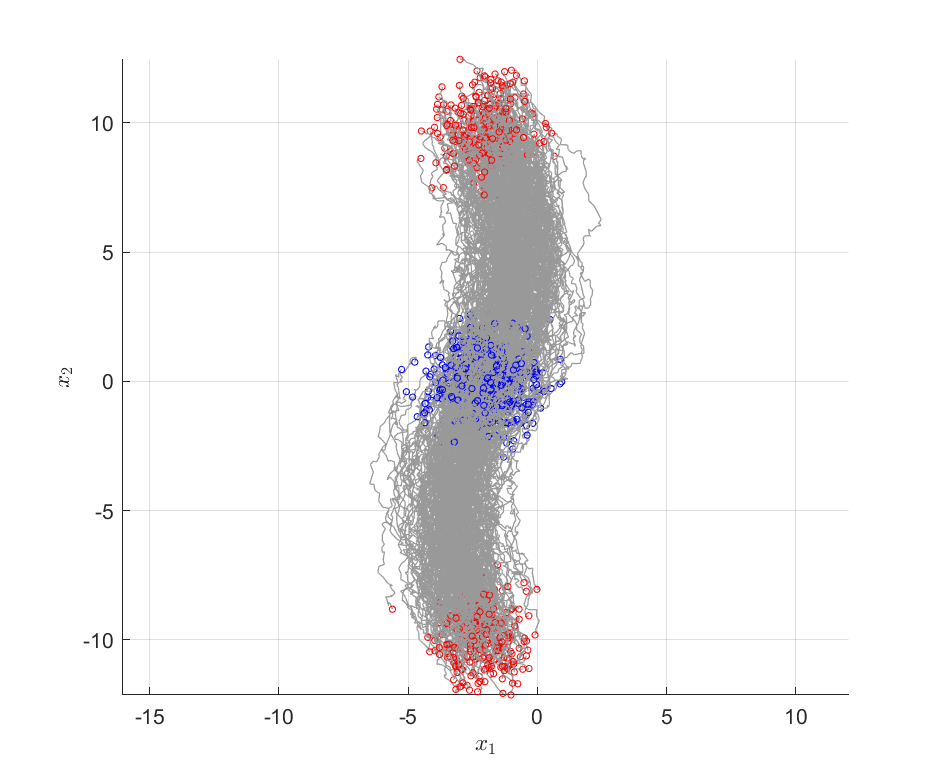}
        \caption{Sample trajectories of the averaged Ornstein-Uhlenbeck process~\eqref{eq:stochastic ensemble of systems} parameterized as~\eqref{eq: 2D-Ornstein-Uhelbeck_parameters} and controlled by the LSTM network. 
        Each gray curve represents one realization of the controlled state trajectory $\{x_{\epsilon}(t)\}_{0 \le t \le 1}$. 
        Blue dots mark samples from the initial distribution $\mu_0$, while red dots indicate the corresponding final states. 
        The LSTM control successfully steers the ensemble toward the target mixture of Gaussians.}
        \label{fig:lstm_control_ou}
    \end{subfigure}
    \hfill
    % --- Right: anti-damped process ---
    \begin{subfigure}[t]{0.46\textwidth}
        \centering
        \includegraphics[width=\textwidth]{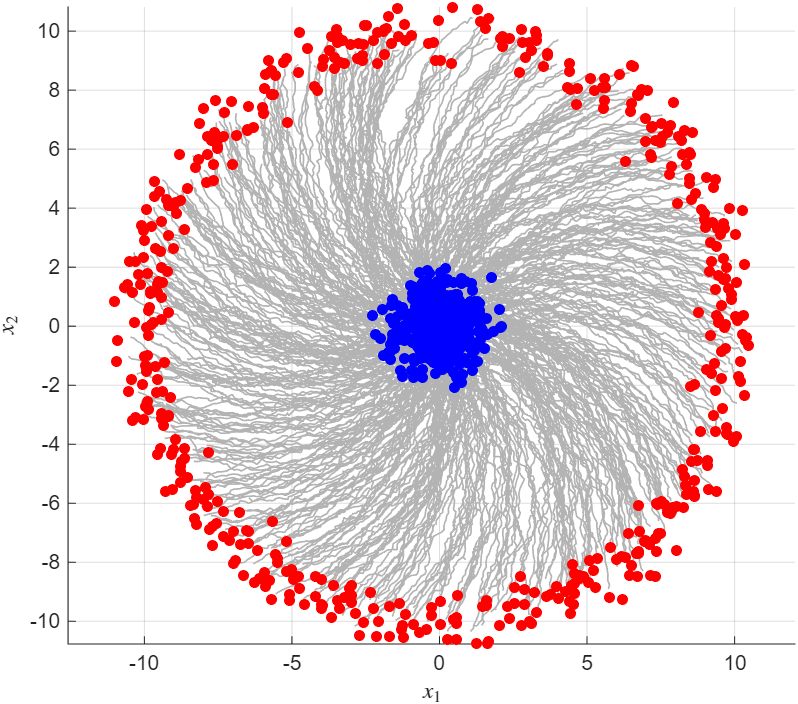}
        \caption{Sample trajectories of the averaged of anti-damped process~\eqref{eq:stochastic ensemble of systems} parameterized by~\eqref{eq: 2D_anti-damped_parameters} under LSTM control. 
        Each gray curve represents one realization of the state trajectory $\{x_{\epsilon}(t)\}_{0 \le t \le 1}$ starting from $\mu_0$.  
        Blue dots denote the initial samples, while red dots mark the endpoints. 
        The learned control expands and organizes the ensemble into the target ring-like distribution.}
        \label{fig:lstm_control_anti_damped}
    \end{subfigure}
    \caption{Comparison of controlled trajectories using LSTM control for two systems: 
    (A) the Ornstein-Uhlenbeck process and (B) the anti-damped process. 
    Both subplots visualize how the learned control steers the stochastic ensemble from $\mu_0$ toward $\mu_f$ under different system dynamics.}
    \label{fig:control_comparison}
\end{figure}

%%%%%%%%%%%%%%%%%%%%%%%%%%%%%%%%%%%%%%
\begin{figure}[h!]
    \centering
    \begin{subfigure}[t]{0.48\textwidth}
        \centering
        \includegraphics[width=\linewidth,height=15.5cm,keepaspectratio]{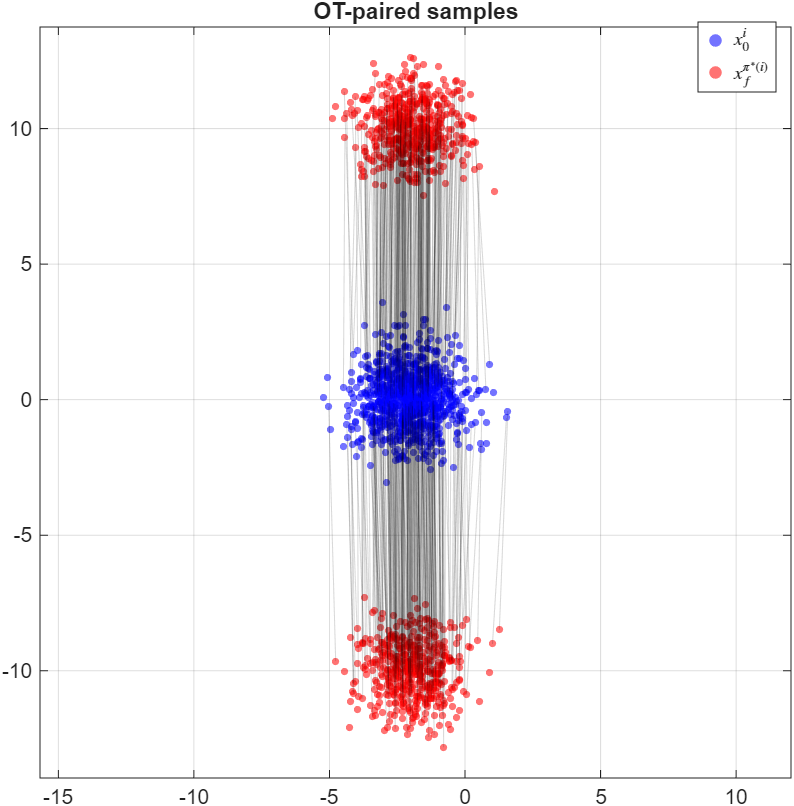}
        \caption{Optimal Transport (OT) pairing between initial samples $x_0^i$ (blue) drawn from a Gaussian distribution and target samples $T(x_0^i)=x_f^{\pi^*(i)}$ (red) drawn from a bimodal Gaussian mixture. 
        Each gray line represents an OT coupling between a sample from the source and one from the target distribution, illustrating the one-to-one transport map that serves as supervision for the control learning.}
        \label{fig:ot_coupling_bimodal}
    \end{subfigure}
    \hfill
    \begin{subfigure}[t]{0.48\textwidth}
        \centering
        \includegraphics[width=\textwidth, height=8.1cm]{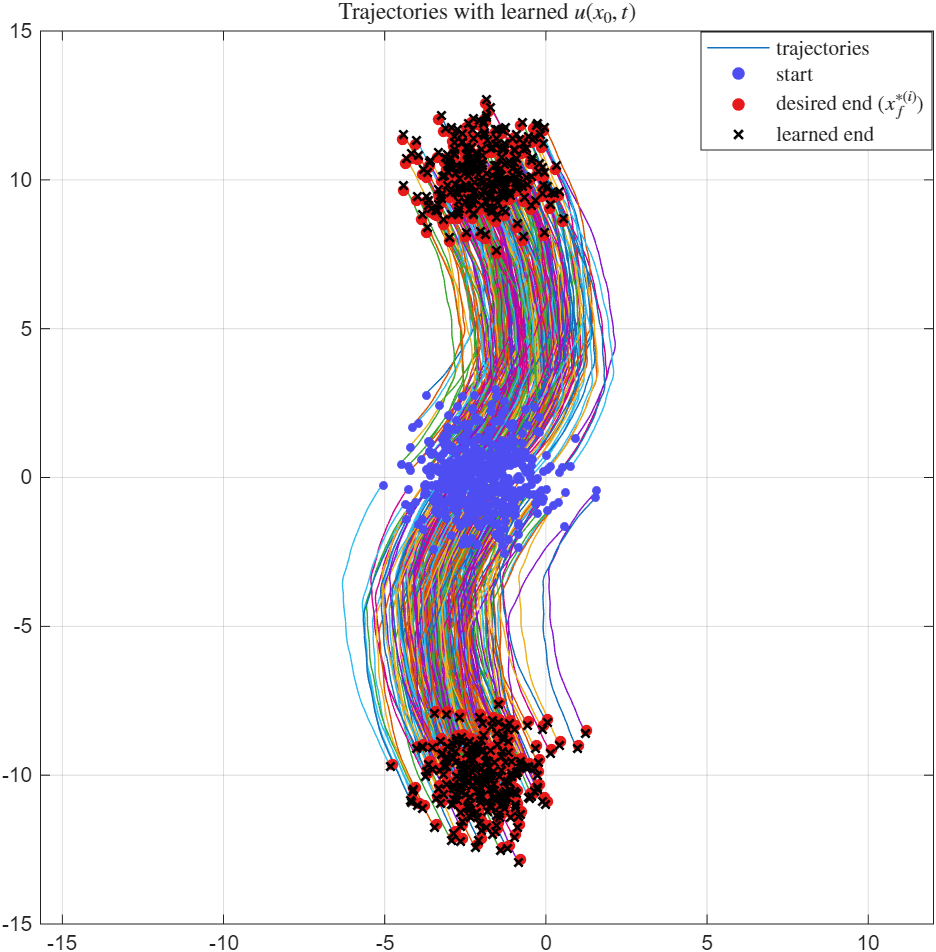}
        \caption{Sample trajectories generated by integrating the learned open-loop control $u(x_0,t)$ obtained from a feedforward neural network and initial-final OT coupling in~(A). 
        Each colored curve corresponds to one controlled trajectory $\{x(t)\}_{0\le t\le 1}$ in~\eqref{eq: noiseless averaged process} starting from $x_0^i$ (blue) and evolving toward the learned terminal points the black crosses and is compared to target samples $T(x_0^i)=x_f^{\pi^*(i)}$ from the OT pairing.}
        \label{fig:learned_open_loop_trajectories}
    \end{subfigure}
    \caption{Comparison between OT-paired samples and controlled trajectories learned from the OT map. 
    (A) OT coupling between initial and target distributions for a bimodal Gaussian target. 
    (B) Trajectories generated by the learned open-loop control $u(x_0,t)$ that dynamically transport samples from the initial Gaussian $\mu_0$ to the bimodal target $\mu_f$.}
    \label{fig:ot_vs_learned_control}
\end{figure}

%%%%%%%%%%%%%%%%%%%%%%%%%%%%%%%%%%%%%%%
\begin{figure}[h!]
    \centering
    % ----- (a) OT Coupling -----
    \begin{subfigure}[t]{0.48\textwidth}
        \centering
        \includegraphics[width=\linewidth,height=15.5cm,keepaspectratio]{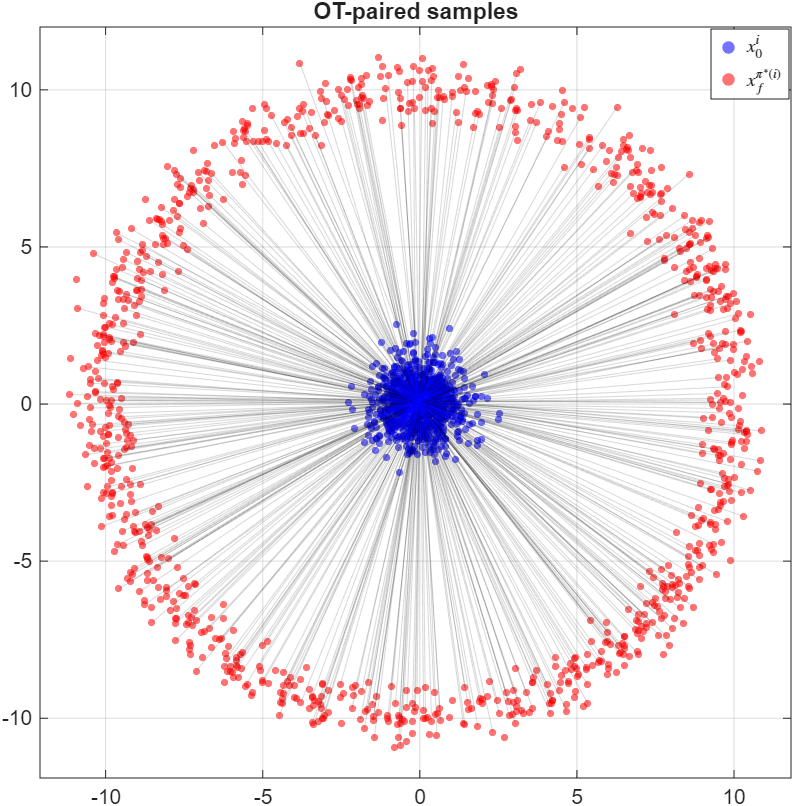}
        \caption{Optimal Transport (OT) coupling between initial samples $x_0^i$ (blue) drawn from the Gaussian source distribution $\mu_0$ and target samples $T(x_0^i)=x_f^{\pi^*(i)}$ (red) drawn from a ring-shaped target distribution $\mu_f$. 
        Each gray line represents the OT-paired correspondence between an initial and target sample, defining the static transport plan for the learning stage.}
        \label{fig:ot_coupling_anti_damped}
    \end{subfigure}
    \hfill
    % ----- (b) Learned control trajectories -----
    \begin{subfigure}[t]{0.48\textwidth}
        \centering
        \includegraphics[width=\textwidth, height=8.1cm]{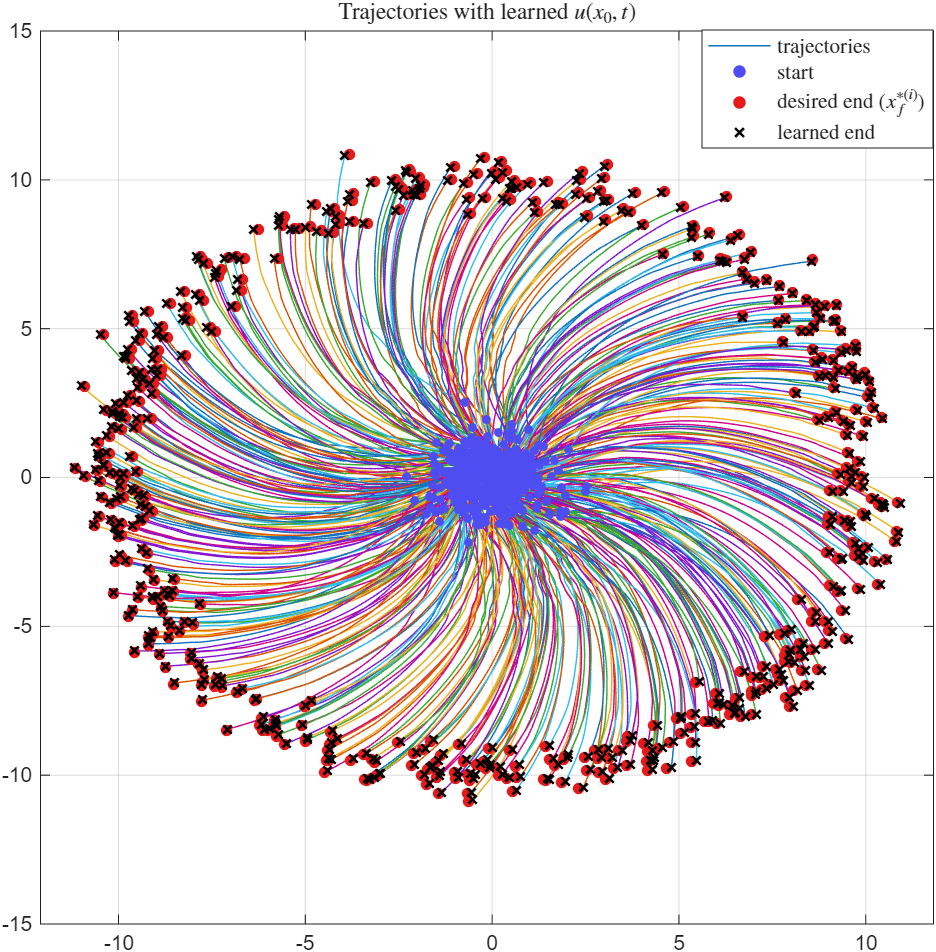}
        \caption{Sample trajectories generated by integrating the learned open-loop control $u(x_0,t)$ derived from the OT coupling in~ Figure~(A). 
        Each colored curve represents one controlled trajectory $\{x(t)\}_{0\le t\le 1}$ in~\eqref{eq: noiseless averaged process} starting from $x_0^i$ (blue) and evolving toward the learned terminal states black crosses and is compared to target samples $T(x_0^i)=x_f^{\pi^*(i)}$. from the OT pairing.}
        \label{fig:learned_open_loop_anti_damped}
    \end{subfigure}
    
    % ----- Overall caption -----
    \caption{Comparison between the OT coupling and the dynamic transport induced by the learned open-loop control for the anti-damped process with $\epsilon=0$. 
    (A) The OT map defines the desired correspondence between initial Gaussian samples and the circular target distribution. 
    (B) The learned control $u(x_0,t)$ realizes this transport dynamically, steering trajectories from $\mu_0$ to $\mu_f$ along smooth paths that respect the system dynamics.}
    \label{fig:ot_vs_learned_control_anti_damped}
\end{figure}

%\begin{remark}
%The employment of LSTM is motivated from the observation that, the local control $u^z_{\varepsilon}(t)$ is non-Markov. A unidirectional LSTM can only summarize the past trajectory and may fail to capture the two-sided dependence required by the conditional law $u_{\epsilon}^z(t)$. Therefore, a BiLSTM architecture effectively approximates the conditional expectation. % By contrast, a BiLSTM processes the sequence forward and backward, so the hidden state at each time $t_j$ encodes both past and future information. This architecture effectively approximates the conditional expectation with respect to the two-sided filtration, which explains why BiLSTM training converges successfully while unidirectional LSTM training does not. 
%\end{remark}

\section{Conclusion}
We have studied a flow matching problem in an ensemble control theoretic framework. We have shown that in the case of a noisy system, this leads to a class of non-Markovian flow matching. To address the memory of the flow matching, we proposed a more amenable numerical methodology (LSTMs) in the learning process.  In the case where there is no noise, one can employ the standard FNN architecture. However, to recover the full final distribution, one must train an open-loop control using deterministic coupling (e.g., an OT permutation plan). As a by-product, we show that under any deterministic coupling one reduces the learning to a one-dimensional regression in time for the gain.

One possible future work is to interpolate using Volterra linear control process.%:
%\begin{equation*}
%X(t) = X(0) + \int_0^t K(t,s) X(s) ds + \int_0^t G(t,s)u(s) ds+ \epsilon\int_0^t G(t,s) dW(s),
%\end{equation*}
%where $K(t,s)$ is a deterministic kernel function that governs the memory effects, $G(t,s)$ is the control or noisy kernel. 
We believe that the control process for the noisy process is also a Volterra process with memory and our numerical approach will be effective in the learning or training process. We state here that Volterra process have significant applications~\cite{NSG-SCM-EWM:71,EZ:21}. Another direction is to extend discrete flow matching \cite{AC-JY-RB-TR-TJ:24,IG-TR-NS-FK-RTC-GS-YL:24} to time-continuous non-Markovian processes on discrete spaces. This offers potential improvements in generative modeling, particularly for tasks with temporal dependencies. % such as language modeling (generating contextually relevant text) and music generation (creating realistic musical sequences).

\end{document}